\documentclass{article}
\usepackage{amsmath,amsthm,amssymb,dsfont,mathtools,bbm,bm,amsbsy}
\usepackage{microtype}
\usepackage{graphicx}
\usepackage{subfigure}
\usepackage{booktabs}
\usepackage{hyperref}
\usepackage{cite}

\usepackage[ruled,vlined]{algorithm2e}
\usepackage[table]{xcolor}

\newtheorem{theorem}{Theorem}
\newtheorem{lem}{Lemma}[section]

\newtheorem{corollary}{Corollary}[theorem]
\theoremstyle{definition}
\newtheorem{definition}{Definition}

\newtheorem*{remark}{Remark}

\newcommand{\mR}{\mathbb{R}}

\newcommand{\mE}{\mathbb{E}}

\newcommand{\mL}{\mathcal{L}}
\newcommand{\bw}{\boldsymbol{w}}
\newcommand{\bz}{\boldsymbol{z}}
\newcommand{\bu}{\boldsymbol{u}}
\newcommand{\Z}{\boldsymbol{Z}}
\newcommand{\mat}[1]{\mathbf{#1}}
\newcommand{\parallelsum}{\mathbin{\!/\mkern-5mu/\!}}

\begin{document}
	
\title{Convergence of a Relaxed Variable Splitting Coarse Gradient Descent Method for Learning
	Sparse Weight Binarized Activation Neural Networks}

\author{Thu Dinh and Jack Xin \thanks{Department of Mathematics, University of California at Irvine, Irvine, CA 92697, USA. Email: (thud2, jack.xin)@uci.edu. }}
\maketitle

\thispagestyle{empty}

\begin{abstract}
\label{sec:abs}
Sparsification of neural networks is one of the effective complexity reduction methods to improve efficiency and generalizability. Binarized activation offers an additional computational saving for inference. 
Due to vanishing gradient issue in training networks with binarized activation, coarse gradient (a.k.a. straight through estimator) is adopted in practice.   
In this paper, we study the problem of coarse gradient descent (CGD) learning of a one hidden layer convolutional neural network (CNN) with binarized activation function and sparse weights. It is known that when the input data is Gaussian distributed, no-overlap one hidden layer CNN with ReLU activation and general weight can be learned by GD in polynomial time at high probability in regression problems with ground truth. We propose a relaxed variable splitting method integrating thresholding and coarse gradient descent. The sparsity in network weight is realized through thresholding during the CGD training process. We prove that under threshholding of $\ell_1, \ell_0,$ and transformed-$\ell_1$ penalties, no-overlap binary activation CNN can be learned with high probability, and the iterative weights converge to a global limit which is a transformation of the true weight under a novel sparsifying operation. We found explicit error estimates of sparse weights from the true weights.  

\end{abstract}

\bigskip

\hspace{.1 in} {\bf Keywords:} Sparse neural networks, sparse penalties, relaxed \\

\hspace{.1 in} variable splitting, thresholding, gradient descent, convergence.
\bigskip

\hspace{.1 in} {\bf Running Title:} Learning sparse neural networks.

\bigskip

\hspace{.1 in} {\bf AMS Subject Classifications:} 90C26, 97R40, 68T05.

\newpage
\setcounter{page}{1}

\section{Introduction}
\label{sec:intro}
Deep neural networks (DNN) have achieved state-of-the-art performance on many machine learning tasks such as 
speech recognition  \cite{Hinton}, computer vision \cite{Krizhevsky}, and natural language processing  \cite{Dauphin}. Training such networks is a problem of minimizing a high-dimensional non-convex and  non-smooth objective function, and is often solved by first-order methods such as stochastic gradient descent (SGD).
Nevertheless, the success of neural network training remains to be understood from a theoretical perspective. Progress has been made in simplified model problems.  \cite{Blum} showed that even training a 3-node neural network is NP-hard, and \cite{Shamir} showed learning a simple one-layer fully connected neural network is hard for some specific input distributions. 
Recently, several works (\cite{Tian}; \cite{Brutzkus}) focused on the geometric properties of loss functions, which is made possible by assuming that the input data distribution is Gaussian. They showed that SGD with random or zero  initialization is able to train a no-overlap neural network in polynomial time.

Another prominent issue is that DNNs contain millions of parameters and lots of redundancies, potentially causing over-fitting and poor generalization \cite{Zhang2016} besides spending unnecessary computational resources. 
One way to reduce complexity is to sparsify the network weights using an empirical technique called pruning \cite{LeCun1989} so that the non-essential ones are zeroed out with minimal loss of performance \cite{Han2015,Ullrich2017,Molch2017}. Recently a surrogate $\ell_0$ regularization approach based on a continuous relaxation of Bernoulli random variables in the distribution sense is introduced with encouraging results on small size image data sets \cite{Welling}. This motivated our work here to study deterministic regularization of $\ell_0$ via its Moreau envelope and related $\ell_1$ penalties in a one hidden layer convolutional neural network model \cite{Brutzkus}. Moreover, we consider binarized activation which further reduces computational  costs \cite{YinBlend}.

The architecture of the network is illustrated in Figure \ref{network_structure} similar to \cite{Brutzkus}. We consider the convolutional setting in which a sparse filter $\bw \in \mR^d$ is shared among different hidden nodes. The input sample is $\Z \in \mR^{k\times d}$. Note that this is identical to the one layer non-overlapping case where the input is $x \in \mR^{kd}$ with $k$ non-overlapping patches, each of size $d$.

\begin{figure}[h]
	\centering
	\includegraphics[scale = 0.3]{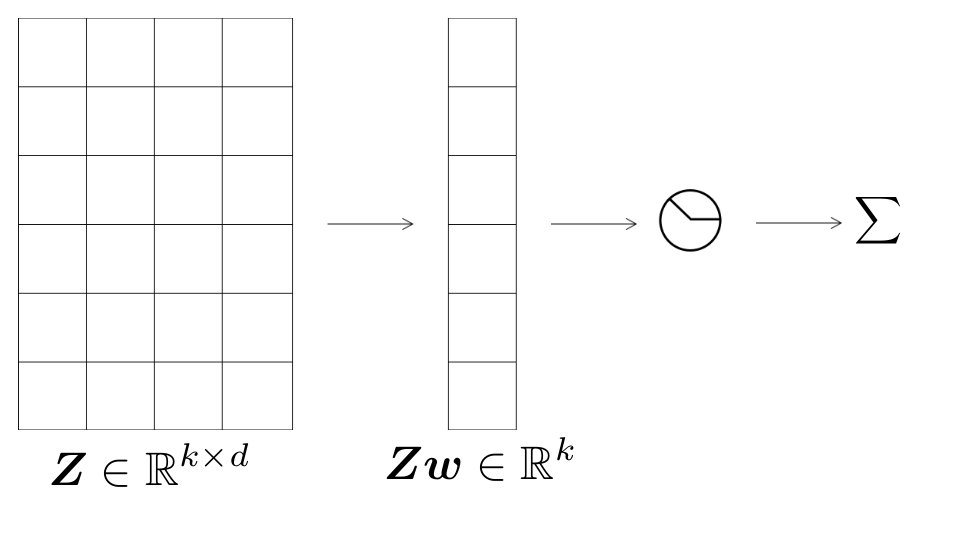}
	\caption{The architecture of a no-overlap neural network}
	\label{network_structure}
\end{figure}

We also assume that the vectors of $\Z$ are i.i.d. Gaussian random vectors with zero mean and unit variance. Let $\mathcal{G}$ denote this distribution. Finally, let $\sigma$ denote the binarized ReLU activation function, $\sigma(z) := \chi_{\{z > 0\}}$ which equals 1 if $z>0$, and 0 otherwise. The output of the network in Figure \ref{network_structure} is given by:
\begin{equation}\label{Eq1}
h(\bw,\Z) = \bm{1}^T \sigma(\Z\bw).
\end{equation}

We address the realizable case, where the response training data is mapped from the input training data $\Z$ by equation \eqref{Eq1} with a ground truth unit weight vector $\bw^*$. The input training data is generated by sampling $m$ training points $\Z^1,..,\Z^m$ from a Gaussian distribution. The learning problem seeks $\bw$ to minimize  the empirical risk function:
\begin{equation}\label{emploss}
l (\bw,\Z) := \frac{1}{m}\sum_{j=1}^{m}\,  (h(\bw,\Z^j)-h(\bw^*,\Z^j))^2
\end{equation}

Due to binarized activation, the gradient of $l$ in $w$ is almost everywhere zero, hence in-effective for descent. Instead, an approximate gradient on the coarse scale, the so called coarse gradient 
(denoted as $\tilde{\nabla}_{\bw} l$) is adopted as proxy and is proved to drive the iterations to global minimum \cite{YinBlend}.

In the limit $m \uparrow \infty$, the empirical risk $l$ converges to the population risk:
\begin{equation}
f(\bw) := \mE_{\Z \sim \mathcal{G}} \left[ (h(\bw,\Z)-h(\bw^*,\Z))^2\right]
\end{equation}
which is more regular in $\bw$ than $l$. However, the ``true gradient'' $\nabla_{\bw} f$ is inaccessible in practice. On the other hand, the coarse gradient $\tilde{\nabla}_{\bw}\, l$ in the limit $m\uparrow \infty$ forms an acute angle with the true gradient \cite{YinBlend}. Hence the expected coarse gradient descent (CGD) essentially minimizes the population risk $f$ as desired.

Our task is to sparsify $\bw$ in CGD.  
We note that the iterative thresholding algorithms (IT) are commonly used for  retrieving sparse signals (\cite{Daubechies,Candes,Blumensath,Blumensath2,Zhang1} and references therein). In high dimensional setting, IT algorithms provide simplicity and low computational cost, while also promote sparsity of the target vector. We shall investigate the convergence of CGD with simultaneous thresholding for the following objective function
\begin{equation}\label{orig}
  \phi(\bw) = f(\bw) + \lambda \, P(\bw)  
\end{equation}
where $f(\bw)$ is the population loss function of the network, and $P$ is $\ell_0$, $\ell_1$, or the transformed-$\ell_1$ (T$\ell_1$) function: a one parameter family of bilinear transformations composed with the absolute value function \cite{Nicolova00,Zhang2}. When acting on vectors, the T$\ell_1$ penalty interpolates $\ell_0$ and $\ell_1$ with thresholding  in closed analytical form for any parameter value \cite{Zhang1}. The $\ell_1$ thresholding function is known as soft-thresholding \cite{Daubechies, Donoho}, and that of $\ell_0$ the  hard-thresholding \cite{Blumensath, Blumensath2}. 
The thresholding part should be properly integrated with CGD to be applicable for learning CNNs. As pointed out in  \cite{Welling}, it is beneficial to attain sparsity during the  optimization (training) process.

{\bf Contribution.}
 We propose a Relaxed Variable Splitting (RVS) approach combining thresholding and CGD for minimizing the following augmented objective function
\begin{equation}\label{Lagr} \mL_\beta(\bu,\bw) = f(\bw) + \lambda\, P(\bu) + \frac{\beta}{2}\, \|\bw-\bu\|^2 \end{equation}
for a positive parameter $\beta$.
We note in passing that minimizing $\mL_\beta$ in $\bu$ recovers the original objective (\ref{orig}) with penalty $P$ replaced by its Moreau envelope \cite{Moreau}. We shall prove that our algorithm (RVSCGD), alternately minimizing $\bu$ and $\bw$, converges for $\ell_0$, $\ell_1$, and T$\ell_1$ penalties to a global limit $(\bar{\bw},\bar{\bu})$ with high probability.  A key estimate is the Lipschitz inequality of the expected coarse gradient (Lemma \ref{properties of coarse gradient}). Then the descent of Lagrangian function (\ref{Lagr}) and the angles between the iterated $\bw$ and $\bw^*$ follows. 
The $\bar{\bw}$ is a novel thresholded version of the true weight $\bw^*$ modulo some normalization. The $\bar{\bu}$ is a sparse approximation of $\bw^*$. {\it To our best knowledge, this result is the first to establish the convergence of CGD for sparse weight binarized activation networks}. In numerical experiments, we observed  that the $\bar{\bu}$ limit of RVSCGD with the $\ell_0$ penalty recovers sparse $\bw^*$ accurately. 

{\bf Outline.} In Section 2, we briefly overview related mathematical results in the study of neural networks and complexity reduction. Preliminaries are in section 3. In Section 4, we state and discuss the main results. The proofs of the main results
are in Section 5, and numerical results in Section 6. The conclusion of this paper is in Section 7 and the acknowledgement is in Section 8.

\section{Related Work}
\label{sec:rel}
In recent years, significant progress has been made in the study of convergence in neural network training. From a theoretical point of view, optimizing (training) neural network is a non-convex non-smooth optimization problem. \cite{Blum, Livni, Shalev}  showed that training a neural network is hard in the worst cases. \cite{Shamir}
showed that if either the target function or input distribution is ``nice", optimization, algorithms used in practice can succeed.
Optimization methods in deep neural networks are often categorized into (stochastic) gradient descent methods and others. 

Stochastic gradient descent methods were first proposed by \cite{Robins}. The popular back-propagation algorithm was introduced in \cite{Rumelhart}. Since then, many well-known SGD methods with adaptive learning rates were proposed and applied in practice, such as the Polyak momentum  \cite{Polyak}, AdaGrad  \cite{ADAGrad}, RMSProp  \cite{RMSProp}, Adam  \cite{Adam}, and AMSGrad  \cite{AMSGrad}.\\ 
The behavior of gradient descent methods in neural networks is better understood when the input has {\it Gaussian} distribution. \cite{Tian} showed that the population gradient descent can recover the true weight vector with random initialization for one-layer one-neuron model. \cite{Brutzkus}
proved that a convolution filter with non-overlapping input can be learned in polynomial time.  \cite{Du1} showed (stochastic) gradient descent with random initialization can learn the convolutional filter in polynomial time and the convergence rate depends on the smoothness of the input distribution and the closeness of patches.  \cite{Du} analyzed the  polynomial convergence guarantee of randomly initialized gradient descent algorithm for learning a one-hidden-layer convolutional neural network. A hybrid projected SGD (so called BinaryConnect) is widely used for training various weight quantized DNNs \cite{BC15,YinObj}. 
Recently, a Moreau envelope based relaxation method (BinaryRelax) is proposed and analyzed to advance 
weight quantization in DNN training \cite{Yin2018}. Also a blended coarse gradient descent method \cite{YinBlend} is introduced to train fully quantized DNNs in weights and activation functions, and overcome vanishing gradients. For earlier work on coarse gradient (a.k.a. straight through estimator), see \cite{hinton12,bnn_16,halfwave17} among others. 

Non-SGD methods for deep learning include the Alternating Direction Method of Multipliers (ADMM) to transform a fully-connected neural network into an equality-constrained problem  \cite{Taylor}; method of auxiliary coordinates (MAC) to replace a nested neural network with a constrained problem without nesting  \cite{Carreira}.  \cite{Zhang} handled deep supervised hashing problem by an ADMM algorithm to overcome vanishing gradients.

For a similar model to (\ref{Lagr}) and treatment in a general context, see \cite{Attouch}; and in image processing, see \cite{Wu}. For analysis and computation on minimization of \eqref{Lagr} to learn a neural network with sparse weight and regular ReLU function, see \cite{Dinh}.

\section{Preliminaries}
\label{sec:prelim}
Consider a non-overlap one layer convolutional network, where the input feature $\Z \in \mR^{k \times d}$ is i.i.d. Gaussian random vector with zero mean and unit variance. Let $\mathcal{G}$ denote this distribution. Let $\sigma$  be the binarized ReLU function:
\[ \sigma(x) = \begin{cases}
0 & \text{if } x \leq 0\\
1 & \text{if } x > 0.
\end{cases} \]
We define the training sample loss by
\begin{equation}\label{sqloss}
l(\bw,\Z) := \frac{1}{2}(\mat{1}^T\sigma(\Z\bw) - \mat{1}^T\sigma(\Z\bw^*))^2
\end{equation}
where $\bw^* \in \mR^d$ is the underlying (non-zero) teaching parameter. Note that (\ref{sqloss}) is invariant under scaling $\bw \to \bw/c$, $\bw^* \to \bw^*/c$, for any scalar $c > 0$. Without loss of generality, we assume $\|\bw^*\|=1$.
Given independent training samples $\{\Z^1,...,\Z^N\}$, the associated empirical risk minimization reads
\begin{equation}\label{emprisk} \min_{\bw \in \mR^d} \frac{1}{N} \sum_{i=1}^N l(\bw,\Z^i). 
\end{equation}
The empirical risk function in  (\ref{emprisk}) is piece-wise constant and has a.e. zero partial $\bw$ gradient. 
If $\sigma$ were differentiable, then back-propagation would rely on:
\begin{equation}\label{partial l} 
\frac{\partial \, l}{\partial \bw}(\bw,\Z) = \Z^T \sigma'(\Z\bw)(\sigma(\Z\bw) - \sigma(\Z\bw^*)).
\end{equation}
However, $\sigma$ has zero derivative a.e., rendering \eqref{partial l} inapplicable. We study the coarse gradient descent with $\sigma'$ in \eqref{partial l} replaced by the (sub)derivative $\mu'$ of the regular ReLU function $\mu(x) := \max(x,0)$. More precisely, we use the following surrogate of $\frac{\partial l}{\partial \bw}(\bw,\Z)$:
\begin{equation}\label{equation g}
g(\bw,\Z) = \sqrt{\frac{2}{\pi}}\Z^T \mu'(\Z\bw)(\sigma(\Z\bw) - \sigma(\Z\bw^*))
\end{equation}
with $\mu'(x) = \sigma(x)$. The constant $\sqrt{\frac{2}{\pi}}$ will be necessary to give a stronger result for our main findings. To simplify our analysis, we let $N\uparrow \infty$ in (\ref{emprisk}), so that its coarse gradient approaches $\mE_{\Z}[g(\bw,\Z)]$. The following lemma asserts that $\mE_{\Z}[g(\bw,\Z)]$ has positive correlation with the true gradient $\nabla f(\bw)$, and consequently, $-\mE_{\Z}[g(\bw,\Z)]$ gives a reasonable descent direction.

\begin{lem}\label{positive correlation}
	If $\theta(\bw,\bw^*) \in (0,\pi)$, and $\|\bw\| \ne 0$, then the inner product between the expected coarse and true gradient w.r.t. $\bw$ is
	\[ \left\langle \mE_{\Z}[g(\bw,\Z)],\nabla f(\bw) \right\rangle
	= \frac{\sin(\theta(\bw,\bw^*))}{4\pi^2\|\bw\|} \geq 0.  \]
\end{lem}

Suppose we want to train the network in a way that $\bw^t$ converges to a limit $\bar{\bw}$ in some neighborhood of $\bw^*$, and we also want to promote sparsity in the limit $\bar{\bw}$. To this end, a natural function to minimize is (for a parameter $\lambda > 0$):
$\phi(\bw) = f(\bw) + \lambda\|\bw\|_1$,
where other choices of a sparse penalty $P$ include $\ell_0$ and $T\ell_1$. Our proposed relaxed variable splitting (RVS) proceeds by first extending $\phi$ 
into a function of two variables
$f(\bw) + \lambda\|\bu\|_1$,
then minimizing the Lagrangian function in \eqref{Lagr} 
alternately in $\bu$ and $\bw$. 
Minimization in $\bu$ is the thresholding operation from penalty $P$, and is in closed form for $\ell_1$, $\ell_0$ and T$\ell_1$. Minimization in $\bw$ is through CGD. The splitting realizes sparsity more effectively than having $P$ under CGD in case of $\ell_1$ (T$\ell_1$), and bypasses the non-existence of gradient in case of $\ell_0$. The resulting RSVCGD Algorithm is as follows:\\

\begin{algorithm}
	\caption{RVSCGD Algorithm}
	\label{RVSM Algorithm}
	\SetAlgoLined
	Initialize $\bu^0,\bw^0$\;
	\While{stopping criteria not satisfied}{
		$\bu^{t+1} \leftarrow \arg\min_{\bu}\,  \mL_\beta\, (\bu,\bw^t)$\\
		$\hat{\bw}^{t+1} \leftarrow \bw^t - \eta\, \mE_{\Z}[\, g(\bw^t,\Z)] - \eta\, \beta\, (\bw^t-\bu^{t+1})$\\
		$ \bw^{t+1} \leftarrow \frac{\hat{\bw}^{t+1}}{\|\hat{\bw}^{t+1}\|}$
	}
	\KwOut{$\bu^t,\bw^t$}
\end{algorithm}

Here the update of $w^t$ has the form $\bw^{t+1} = C^t(\bw^t - \eta\mE_{\Z}[g(\bw^t,\Z)] - \eta\beta(\bw^t-\bu^{t+1}))$, where $C^t$ is some normalization constant. This normalization process is unique to our proposed algorithm, and is distinct from other common descent algorithms, for example ADMM, where the update of $\bw$ has the form $\bw^{t+1} \leftarrow \arg\min_{\bw} \mL_\beta(\bu^{t+1},\bw,\bz^t)$ and $\bz^t$ is the Lagrange multiplier. Since $f$ is non-convex and only Lipschitz differentiable away from zero, convergence analysis of ADMM in this case is beyond the current theory \cite{Wang}. Here we  circumvent the problem by updating $\bw$ via CGD and then normalizing.

\begin{definition}\label{TL1Definition}
The transformed $\ell_1$ (T$\ell_1$) penalty function on $x \in \mR^d$ is
$P_a(x) := \sum_{i=1}^d \rho_a(x_i)$, where  
$\rho_a(x) = \frac{(a+1)|x|}{a+|x|}$,
for a parameter $a \in (0,+\infty)$. 
 
\end{definition}
By varying 'a', T$\ell_1$ interpolates $l_0$ and $l_1$.

\section{Main Results}
\label{sec:mainresult}
\begin{theorem}\label{Main result}
Suppose that the initialization and penalty parameters of the RVSCGD algorithm satisfy:\\
(i) $\theta(\bw^0,\bw^*) \leq \pi-\delta$, for some $\delta > 0$,\\
(ii) $\beta \leq \frac{k\sin\delta}{2\pi}$, and $\lambda < \frac{k}{2\pi\sqrt{d}}$;\\
and that the learning rate $\eta$ 
is small so that 
\[\eta \,\|\mE_{\Z}[g(\bw^t,\Z)] + \beta\, (\bw^t-\bu^{t+1})\| \leq \frac{1}{2},\; \; \forall t;
\]
and $\eta \leq \min\left\{ \frac{1}{\beta+L}, \frac{2\pi}{k} \right\}$, where $L$ is the Lipschitz constant in Lemma \ref{properties of coarse gradient}.
Then the Lagrangian $\mL_\beta(\bu^t,\bw^t)$ with $\ell_1$ penalty is monotonically decreasing, and $(\bu^t,\bw^t)$ converges to a limit point $(\bar{\bu},\bar{\bw})$. Let $\theta := \theta(\bar{\bw},\bw^*)$ and $\gamma := \theta(\bar{\bu},\bar{\bw})$, then $\theta < \delta$, and the critical point $(\bar{\bu},\bar{\bw})$ satisfies
\begin{equation}\label{limit point equality}
\bar{\bw} - \frac{2\pi}{k}\beta(\bar{\bw}-S_{\lambda/\beta}(\bar{\bw})) = C\bar{\bw}
\end{equation}
where $S_{\lambda/\beta}$ is the soft-thresholding operator of $\ell_1$, for some positive constant $C$ such that $C \leq \frac{k}{k-2\pi\lambda\sqrt{d}}$; and 
\begin{equation}\label{limit point bound}
\|\bw^*-\bar{\bw}\, \| \leq \frac{1}{2}\sin\delta\sin\gamma,
\end{equation}
\begin{equation}\label{limit point bound2}
\|\bw^*-\bar{\bu}\, \| \leq \frac{1}{2}\sin\delta\sin\gamma
+ {\lambda \over \beta} \sqrt{d}. 
\end{equation}
\end{theorem}

\begin{remark}
The sign of $(\bar{\bw}-S_{\lambda/\beta}(\bar{\bw}))$ agrees with $\bar{\bw}$. Thus $\bar{w}$ is a soft-thresholded version of $\bw^*$, after some normalization. The assumption on $\eta$ is reasonable, as will be shown below: $\|\mE_{\Z}[g(\bw^t,\Z)]\|$ is bounded away from zero, and thus $\|\mE_{\Z}[g(\bw^t,\Z)] + \beta(\bw^t-\bu^{t+1})\|$ is also bounded.
\end{remark}
\smallskip

\begin{corollary}\label{Corollary}
Suppose that the initialization of the RVSCGD algorithm satisfies the conditions in Theorem \ref{Main result}, and that the $\ell_1$ penalty is replaced by $\ell_0$ or T$\ell_1$.
Then the RVSCGD iterations converge to a limit point $(\bar{\bu},\bar{\bw})$  satisfying equation \eqref{limit point equality} with $\ell_0$'s hard thresholding operator \cite{Blumensath2} or T$\ell_1$ thresholding \cite{Zhang1} replacing $S_{\lambda/\beta}$, and similar bounds \eqref{limit point bound}-\eqref{limit point bound2} hold.
\end{corollary}

\section{Proof of Main Results}
\label{sec:proof}
\subsection{Proof of Theorem \ref{Main result}}
The following Lemmas give the properties of the coarse gradient, as well as an outline for the proof of Theorem \ref{Main result}. The detail of each key step can be found in the proof of the corresponding Lemma. 
\begin{lem}\label{gradient lemma}
If every entry of $\Z$ is i.i.d. sampled from $\mathcal{N}(0,1), \|\bw^*\|=1$, and $\|\bw\| \ne 0$, then the true gradient of the population loss $f(\bw)$ is
\begin{equation}\label{population loss gradient}
\nabla f(\bw) = \frac{-k}{2\pi\|\bw\|}
\frac{\left( \mat{I} - \frac{\bw\bw^T}{\|\bw\|^2} \right)\bw^*}
{\left\|\left( \mat{I} - \frac{\bw\bw^T}{\|\bw\|^2} \right)\bw^*\right\|},
\end{equation}	
for $\theta(\bw,\bw^*) \in (0,\pi)$; and the expected coarse gradient w.r.t. $\bw$ is
\begin{align*}
&\mE_{\Z}[g(\bw,\Z)] \\
=& \frac{k}{\pi}\left[ \frac{\bw}{\|\bw\|} - \cos\left(\frac{\theta(\bw,\bw^*)}{2}\right)
\frac{\frac{\bw}{\|\bw\|}+\bw^*}{\left\|\frac{\bw}{\|\bw\|}+\bw^*\right\|} \right]
\stepcounter{equation}\tag{\theequation}\label{coarse gradient}
\end{align*}
\end{lem}

\begin{lem}\label{properties of true gradient}
(Properties of true gradient)\\
Given $\bw_1,\bw_2$ with $\min\{\|\bw_1\|,\|\bw_2\|\} = c >0$ and $\max\{\|\bw_1\|,\|\bw_2\|\} = C$, there exists a constant $L_f>0$ depends on $c$ and $C$ such that
\begin{equation*}
\|\nabla f(\bw_1) - \nabla f(\bw_2)\| \leq L_f\|\bw_1-\bw_2\|
\end{equation*}
Moreover, we have
\begin{equation*}
f(\bw_2) \leq f(\bw_1) + \langle \nabla f(\bw_1), \bw_2-\bw_1 \rangle + \frac{L_f}{2} \|\bw_2-\bw_1\|^2.
\end{equation*}
\end{lem}

Lemmas \ref{gradient lemma}, \ref{properties of true gradient} follow directly from \cite{YinBlend}.

\begin{lem}\label{properties of coarse gradient}
(Properties of expected coarse gradient)\\
If $\bw_1,\bw_2$ satisfy $\|\bw_1\|, \|\bw_2\|=1$, and $\theta(\bw_1,\bw^*),\theta(\bw_2,\bw^*) \in (0,\pi)$ with $\theta(\bw_2,\bw^*) \leq \theta(\bw_1,\bw^*)$, then there exists a constant $K = \frac{k}{2\pi}$ such that
\begin{equation}
\| \mE_{\Z}[g(\bw_1,\Z)] - \mE_{\Z}[g(\bw_2,\Z)] \| \leq K \|\bw_1-\bw_2\|
\end{equation}
Moreover, there exists a constant $L = \frac{k}{4\pi}$ such that
\begin{align*}
&f(\bw_2) - f(\bw_1)\\
\leq& \langle \nabla \mE_{\Z}[g(\bw_1,\Z)], \bw_2-\bw_1 \rangle + \frac{L}{2} \|\bw_2-\bw_1\|^2.
\stepcounter{equation}\tag{\theequation}\label{descent condition}
\end{align*}
\end{lem}

\begin{remark}
Here we show the coarse gradient satisfies the descent condition under very specific conditions. As shown in Lemma \ref{angle decreases lemma}, we have $\theta^{t+1}\leq\theta^t$; and the normalization of our RVSCGD model guarantees $\|\bw^{t+1}\|=\|\bw^t\|=1$. Thus the result of Lemma \ref{properties of coarse gradient} is sufficient for our proof.
\end{remark}

\begin{proof}
First suppose $\|\bw_1\| = \|\bw_2\| = 1$. By Lemma 5.3 of \cite{YinBlend}, we have
\begin{align*}
&\mE_{\Z}[g(\bw_j,\Z)] \\
=& \frac{k}{\pi}\left[\bw_j - \cos\left(\frac{\theta(\bw_j,\bw^*)}{2}\right)
\frac{\bw_j+\bw^*}{\left\|\bw_j+\bw^*\right\|} \right]
\end{align*}
for $j=1,2$. Consider the plane formed by $\bw_j$ and $\bw^*$, since $\|\bw^*\| = 1$, we have an equilateral triangle formed by $\bw_j$ and $\bw^*$ (See Fig. \ref{rhombus}).\\

\begin{figure}
	\centering
	\includegraphics[scale = 0.4]{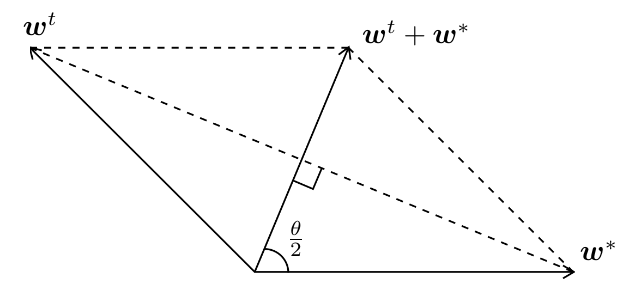}
	\caption{Geometry of $\bw^t$ and $\bw^*$ when $\|\bw^t\|=\|\bw^*\|=1$.}
	\label{rhombus}
\end{figure}

Simple geometry shows
\[ \cos\left(\frac{\theta(\bw_j,\bw^*)}{2}\right)
= \frac{\frac{1}{2}\|\bw_j+\bw^*\|}{\|\bw^*\|} = \frac{1}{2}\|\bw_j+\bw^*\| \]
Thus the expected coarse gradient simplifies to
\begin{align*}
\mE_{\Z}[g(\bw_j,\Z)] &= \frac{k}{\pi}\left[\bw_j - \frac{\bw_j+\bw^*}{2} \right]\\
&= \frac{k}{2\pi}\bw_j - \frac{k}{2\pi}\bw^*
\stepcounter{equation}\tag{\theequation}\label{coarse gradient formula}
\end{align*}
which implies
\begin{equation}\label{lipschitz coarse gradient}
\| \mE_{\Z}[g(\bw_1,\Z)] - \mE_{\Z}[g(\bw_2,\Z)] \| \leq K \|\bw_1-\bw_2\|
\end{equation}
with $K = \frac{k}{2\pi}$. The first claim is proved.\\

It remains to show the gradient descent inequality. By \cite{YinBlend}, we have
\[ f(\bw) = \frac{1}{8}\left[ \mat{1}^T(I + \mat{11}^T)\mat{1}-2\mat{1}^T\left(
\left(1-\frac{2}{\pi}\theta(\bw,\bw^*)\right)I+\mat{11}^T \right)\mat{1} + \mat{1}^T(I + \mat{11}^T)\mat{1} \right] \]
Let $\theta_1 = \theta(\bw_1,\bw^*), \theta_2 = \theta(\bw_2,\bw^*)$. Then
\begin{align*}
f(\bw_2)-f(\bw_1) &= \frac{1}{4}\left[-\mat{1}^T\left(
\left(1-\frac{2}{\pi}\theta_2\right)I+\mat{11}^T \right)\mat{1}
+\mat{1}^T\left(
\left(1-\frac{2}{\pi}\theta_1\right)I+\mat{11}^T \right)\mat{1}  \right]\\
&= \frac{1}{4} \left[ \mat{1}^T\left( \left(\frac{2}{\pi}\theta_2 - \frac{2}{\pi}\theta_1\right)I\right)\mat{1} \right]\\
&= \frac{k}{2\pi} (\theta_2-\theta_1)
\end{align*}
We will show $f(\bw_2)-f(\bw_1) \leq \langle \mE_{\Z}[g(\bw_1,\Z)],\bw_2-\bw_1 \rangle + L\|\bw_2-\bw_1\|^2$ for $\|\bw_1\|=\|\bw_2\|=1$ and $\theta_2 \leq \theta_1$. By equation \eqref{coarse gradient formula},
\[\mE_{\Z}[g(\bw_1,\Z)] = \frac{k}{2\pi} \left(\bw_1-\bw^*\right) \] 
It remains to show
\[ \frac{k}{2\pi} (\theta_2-\theta_1) \leq \left\langle \frac{k}{2\pi} \left(\bw_1-\bw^*\right), \bw_2-\bw_1 \right\rangle + L\|\bw_2-\bw_1\|^2 \]
or there exists a constant $K_1$ such that
\[ \theta_2-\theta_1 \leq \langle \bw_1-\bw^*,\bw_2-\bw_1 \rangle + K_1\|\bw_2-\bw_1\|^2 \]
Notice that by writing $K_1 = \frac{1}{2}+K_2$, we have
\begin{align*}
&\langle \bw_1-\bw^*,\bw_2-\bw_1 \rangle + K_1\|\bw_2-\bw_1\|^2\\
&= \langle \bw_1-\bw^*,\bw_2-\bw_1 \rangle + K_1\langle\bw_2-\bw_1,\bw_2-\bw_1\rangle\\
&= \langle \bw_1-\bw^*,\bw_2-\bw_1 \rangle + \frac{1}{2} \langle\bw_2-\bw_1,\bw_2-\bw_1\rangle + K_2\|\bw_2-\bw_1\|^2\\
&= \langle \frac{1}{2}\bw_1+\frac{1}{2}\bw_2-\bw^*, \bw_2-\bw_1 \rangle + K_2 \|\bw_2-\bw_1\|^2\\
&= \langle -\bw^*, \bw_2-\bw_1 \rangle + \frac{1}{2}\langle \bw_1+\bw_2,\bw_2-\bw_1\rangle + K_2\|\bw_2-\bw_1\|^2\\
&= \langle -\bw^*, \bw_2-\bw_1 \rangle + K_2 \|\bw_2-\bw_1\|^2
\end{align*}
where the last equality follows since $\|\bw_1\|=\|\bw_2\|=1$ implies $\langle \bw_1+\bw_2,\bw_2-\bw_1\rangle=0$. On the other hand,
\[ \langle -\bw^*,\bw_2-\bw_1 \rangle  = -\|\bw^*\|\|\bw_2\|\cos\theta_2 + \|\bw^*\|\|\bw_1\|\cos\theta_1
= \cos\theta_1-\cos\theta_2 \]
so it suffices to show there exists a constant $K_2$ such that
\[ \theta_2+\cos\theta_2 - \theta_1 - \cos\theta_1 \leq K_2\|\bw_2-\bw_1\|^2 \]
Notice the function $\theta \mapsto \theta+\cos\theta$ is monotonically increasing on $[0,\pi]$. For $\theta_1,\theta_2 \in [0,\pi]$ with $\theta_2\leq\theta_1$, the LHS is non-positive, and the inequality holds. Thus one can take $K_2=0, K_1 = \frac{1}{2}$, and $L=\frac{k}{4\pi}$.
\end{proof}

\begin{lem}\label{angle decreases lemma}
(Angle Descent)\\
Let $\theta^t := \theta(\bw^t,\bw^*)$. If the initialization of the RVSCGD algorithm satisfies 
\[\theta^0 \leq \pi-\delta, \quad \text{and} \quad \beta \leq \frac{k\sin\delta}{2\pi},\] 
then
\begin{equation}
\theta^{t+1} \leq \theta^t.
\end{equation}
\end{lem}

\begin{proof}
Due to normalization in the RVSCGD algorithm, $\|\bw^t\| = 1$ for all $t$. By equation \eqref{coarse gradient formula}, we have
\begin{equation*}
\bw^t - \eta\mE_{\Z}[g(\bw^t,\Z)] = \left(1-\eta\frac{k}{2\pi}\right) \bw^t + \eta\frac{k}{2\pi}\bw^*
\end{equation*}
and the update of $\bu$ is the well-known soft-thresholding of $\bw$  \cite{Donoho,Daubechies}:
\begin{equation*}
\bu^{t+1} = \arg\min_{\bu} \mL_\beta(\bu,\bw^t) = S_{\lambda/\beta}(\bw^t)
\end{equation*}
where $S_{\lambda/\beta}(\cdot)$ is the soft-thresholding operator:
\begin{equation*}
S_{\lambda/\beta}(x) = \begin{cases}
x - \lambda/\beta, & x > \lambda/\beta\\
0, & |x| \leq \lambda/\beta\\
x + \lambda/\beta, & x < -\lambda/\beta
\end{cases}
\end{equation*}
and $S_{\lambda/\beta}(\bw)$ applies the thresholding to each component of $\bw$. Then the update of $\bw$ has the form
\begin{equation*}
\bw^{t+1} = C^t\bw^t + \eta\frac{k}{2\pi}\bw^* + \eta\beta\bu^{t+1}
\end{equation*}
for some constant $C^t>0$. Suppose the initialization satisfies $\theta(\bw^0,\bw^*) \leq \pi-\delta$, for some $\delta > 0$. It suffices to show that if $\theta^t \leq \pi-\delta$, then $\theta^{t+1} \leq \pi-\delta$. To this end, since $\bu^{t+1} = S_{\lambda/\beta}(\bw^t)$, we have $\theta(\bw^t,\bu^{t+1}) \leq \frac{\pi}{2}$. Consider the worst case scenario: 
 $\bw^t,\bw^*,\bu^{t+1}$ are co-planar with $\theta(\bu^{t+1},\bw^t) = \frac{\pi}{2}$, and $\bw^*,\bu^{t+1}$ are on two sides of $\bw^t$ (See Fig. \ref{worst case}). \\
\begin{figure}
	\centering
	\includegraphics[scale = 0.5]{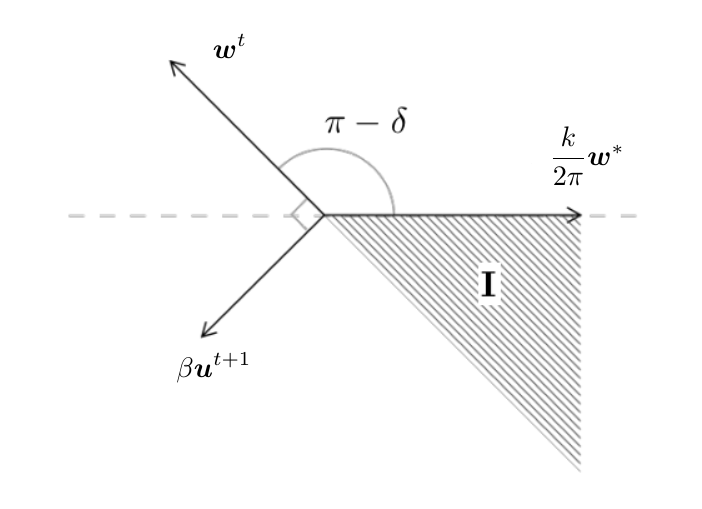}
	\caption{Worst case of the update on $w^t$}
	\label{worst case}
\end{figure}
We need $\frac{k}{2\pi} \bw^* + \beta \bu^{t+1}$ to be in region I. This condition is satisfied when $\beta$ is small such that
\[ \sin\delta \geq \frac{\beta\|\bu^{t+1}\|}{\frac{k}{2\pi}\|\bw^*\|}
= \frac{2\pi\beta\|\bu^{t+1}\|}{k} \]
or
\[ \beta \leq \frac{k\sin\delta}{2\pi\|\bu^{t+1}\|} \]
Since $\|u^{t+1}\|\leq 1$, it suffices to have $ \beta \leq \frac{k\sin\delta}{2\pi}$.
\end{proof}
\begin{lem}\label{lagrangian decreases lemma}
(Lagrangian Descent)\\
If the initialization of the RVSCGD algorithm satisfies 
$\eta \leq \frac{1}{\beta+L}$, 
where $L$ is the Lipschitz constant in Lemma \ref{properties of coarse gradient}, then
\begin{equation}\label{lagrangian decreases}
\mL_\beta(\bu^{t+1},\bw^{t+1}) \leq \mL_\beta(\bu^t,\bw^t).
\end{equation}
\end{lem}
\begin{proof}
By definition of the update on $\bu$, we have $\mL_\beta(\bu^{t+1},\bw^t) \leq \mL_\beta(\bu^t,\bw^t)$. It remains to show $\mL_\beta(\bu^{t+1},\bw^{t+1}) \leq \mL_\beta(\bu^{t+1},\bw^t)$. First notice that since
\[ \bw^{t+1} = C^t(\bw^t - \eta\mE_{\Z}[g(\bw^t,\Z)] - \eta\beta(\bw^t-\bu^{t+1}))\] 
where $C^t>0$ is the normalizing constant, thus
\[ \mE_{\Z}[g(\bw^t,\Z)] = \frac{1}{\eta}\left(\bw^t-\frac{\bw^{t+1}}{C^t}\right) - \beta(\bw^t-\bu^{t+1}) \]
For a fixed $\bu:=\bu^{t+1}$ we have
\begin{align*}
&\mL_\beta(\bu,\bw^{t+1}) - \mL_\beta(\bu,\bw^t)\\
=& f(\bw^{t+1}) - f(\bw^t) 
+ \frac{\beta}{2}\left(\|\bw^{t+1}-\bu\|^2-\|\bw^t-\bu\|^2\right)\\
\leq& \langle \mE_{\Z}[g(\bw^t,\Z)], \bw^{t+1}-\bw^t \rangle + \frac{L}{2}\|\bw^{t+1}-\bw^t\|^2\\
+& \frac{\beta}{2}\left(\|\bw^{t+1}-\bu\|^2-\|\bw^t-\bu\|^2\right)\\
=&\frac{1}{\eta}\langle \bw^t - \frac{\bw^{t+1}}{C^t}, \bw^{t+1}-\bw^t \rangle
- \beta \langle \bw^t-\bu	, \bw^{t+1}-\bw^t \rangle \\
+& \frac{L}{2}\|\bw^{t+1}-\bw^t\|^2 + \frac{\beta}{2}\left(\|\bw^{t+1}-\bu\|^2-\|\bw^t-\bu\|^2\right)\\
=& \frac{1}{\eta}\langle \bw^t - \frac{\bw^{t+1}}{C^t}, \bw^{t+1}-\bw^t \rangle
+ \left( \frac{L}{2} + \frac{\beta}{2} \right)\|\bw^{t+1}-\bw^t\|^2\\
+& \frac{\beta}{2}\|\bw^{t+1}-\bu\|^2 - \frac{\beta}{2}\|\bw^t-\bu\|^2\\
-& \beta \langle \bw^t-\bu	, \bw^{t+1}-\bw^t \rangle
- \frac{\beta}{2} \|\bw^{t+1}-\bw^t\|^2\\
=& \frac{1}{\eta}\langle \bw^t - \frac{\bw^{t+1}}{C^t}, \bw^{t+1}-\bw^t \rangle
+ \left( \frac{L}{2} + \frac{\beta}{2} \right)\|\bw^{t+1}-\bw^t\|^2
\end{align*}
Since $\|\bw^t\|,\|\bw^{t+1}\| = 1$, we know $(\bw^{t+1}-\bw^t)$ bisects the angle between $\bw^{t+1}$ and $-\bw^t$. The assumption $\|\eta\mE_{\Z}[g(\bw^t,\Z)] + \eta\beta(\bw^t-\bu^{t+1})\| \leq \frac{1}{2}$ guarantees $\frac{2}{3} \leq C^t \leq 2$ and $\theta(-\bw^t,\bw^{t+1}) < \pi$. It follows that $\theta(\bw^{t+1}-\bw^t,\bw^t)$ and $\theta(\bw^{t+1}-\bw^t,\bw^{t+1})$ are strictly less than $\frac{\pi}{2}$. On the other hand, $\left(\frac{\bw^{t+1}}{C^t}-\bw^t\right)$ also lies in the plane bounded by $\bw^{t+1}$ and $-\bw^t$. Therefore 
\[ \theta\left(\frac{\bw^{t+1}}{C^t}-\bw^t, \bw^{t+1}-\bw^t\right) < \frac{\pi}{2}. \]
This implies 
$\langle \frac{\bw^{t+1}}{C^t}-\bw^t, \bw^{t+1}-\bw^t \rangle \geq 0$. Moreover, when $C^t \geq 1$:
\begin{align*}
&\langle \frac{\bw^{t+1}}{C^t}-\bw^t, \bw^{t+1}-\bw^t \rangle \\
=& \langle \frac{\bw^{t+1}}{C^t}-\frac{\bw^t}{C^t}, \bw^{t+1}-\bw^t \rangle
- \langle \frac{C^t-1}{C^t} \bw^t, \bw^{t+1}-\bw^t \rangle \\
\geq& \frac{1}{C^t}\|\bw^{t+1}-\bw^t\|^2
\end{align*}
And when $\frac{2}{3} \leq C^t \leq 1$:
\begin{align*}
&\langle \frac{\bw^{t+1}}{C^t}-\bw^t, \bw^{t+1}-\bw^t \rangle \\
=& \langle {t+1}-\bw^t, \bw^{t+1}-\bw^t \rangle
+ \langle \frac{1-C^t}{C^t} \bw^{t+1}, \bw^{t+1}-\bw^t \rangle \\
\geq& \|\bw^{t+1}-\bw^t\|^2
\end{align*}
Thus we have
\begin{align*}
&\mL_\beta(\bu,\bw^{t+1}) - \mL_\beta(\bu,\bw^t)\\
\leq& \frac{1}{\eta}\langle \bw^t - \frac{\bw^{t+1}}{C^t}, \bw^{t+1}-\bw^t \rangle
+ \left( \frac{L}{2} + \frac{\beta}{2} \right)\|\bw^{t+1}-\bw^t\|^2\\
\leq& \left( \frac{L}{2} + \frac{\beta}{2} - \frac{1}{\eta C^t}
\chi_{\{C^t \geq 1\}} - \frac{1}{\eta} \chi_{\{\frac{2}{3} \leq C^t \leq 1\}} \right)\|\bw^{t+1}-\bw^t\|^2
\end{align*}
Therefore, if $\eta$ is small so that $\eta \leq \frac{2}{C^t(\beta+L)}$ and $\eta \leq \frac{2}{\beta + L}$, the update on $\bw$ will decrease $\mL_\beta$. Since $C^t \leq 2$, the condition is satisfied when $\eta \leq \frac{1}{\beta+L}$.
\end{proof}

\begin{lem}\label{properties of limit point}
(Properties of limit point)\\
If the initialization of the RVSCGD  algorithm satisfies\\
(i) $\theta(\bw^0,\bw^*) \leq \pi-\delta$, for some $\delta>0$\\
(ii) $\lambda$ is small such that $\frac{2\pi}{k}\lambda\sqrt{d} < 1$, $\beta$ is small such that $\beta \leq \frac{k\sin\delta}{2\pi}$\\
(iii) $\eta$ is small such that $\eta\frac{k}{2\pi} < 1$\\
Let $\theta := \theta(\bar{\bw},\bw^*)$ and $\gamma := \theta(\bar{\bu},\bar{\bw})$, then $(\bu^t,\bw^t)$ converges to a critical point $(\bar{\bu},\bar{\bw})$ such that
\[ \theta < \delta \;\;{\rm and }\;\;
 \|\bw^*-\bar{\bw}\| \leq \frac{1}{2}\sin\delta\sin\gamma. \]
\end{lem}

\begin{proof}
Since $\mathcal{L}_\beta(\bu^t,\bw^t)$ is non-negative, by Lemma \ref{angle decreases lemma}, \ref{lagrangian decreases lemma}, $\mathcal{L}_\beta$ converges to some limit $\mathcal{L}$. This implies $(\bu^t,\bw^t)$ converges to some stationary point $(\bar{\bu},\bar{\bw})$. By the update of $\bw^t$, we have
\begin{equation}\label{EquilibriumCondition}
\bar{\bw} = \bar{C}(c_1 \bar{\bw} + \eta c_2\bw^* + \eta\beta \bar{\bu})
\end{equation}
for some constant $\bar{C}, c_1, c_2 > 0$,
where $c_2 = \frac{k}{2\pi}$, $c_1>0$ due to condition $(iii)$, and $\bar{\bu} = S_{\lambda/\beta} (\bar{\bw})$. For expression (\ref{EquilibriumCondition}) to hold, we need
\begin{equation}\label{ParallelCondition}
c_2 \bw^* + \beta\bar{\bu} \parallelsum \bar{\bw}
\end{equation}
Expression (\ref{ParallelCondition}) implies $\bar{\bw},\bar{\bu}$, and $\bw^*$ are co-planar. Let $\gamma := \theta(\bar{\bu},\bar{\bw})$. From expression (\ref{ParallelCondition}), and the fact that $\|\bar{\bw}\|=\|\bw^*\|=1$, we have
\begin{align*}
&(\langle c_2\bw^*+ \beta\bar{\bu}, \bar{\bw} \rangle)^2 = \|c_2\bw^*+\beta\bar{\bu}\|^2\|\bar{\bw}\|^2\\
&(c_2\langle \bw^*,\bar{\bw}\rangle + \beta\langle\bar{\bu},\bar{\bw}\rangle)^2 = \langle c_2\bw^*+\beta\bar{\bu}, c_2\bw^*+\beta\bar{\bu} \rangle
\end{align*}
thus
\begin{align*}
&c_2^2\cos^2\theta + 2c_2\beta\|\bar{\bu}\|\cos\theta\cos\gamma + \beta^2\|\bar{\bu}\|^2\cos^2\gamma\\
=& c_2^2 + 2c_2\beta\|\bar{\bu}\|\cos(\theta+\gamma) + \beta^2\|\bar{\bu}\|^2
\end{align*}
Recall $\cos(a+b) = \cos a\cos b - \sin a\sin b$. Thus
\begin{align*}
&c_2^2\sin^2\theta - 2c_2\beta\|\bar{\bu}\|\sin\theta\sin\gamma + \beta^2\|\bar{\bu}\|^2\sin^2\gamma = 0\\
&(c_2\sin\theta - \beta\|\bar{\bu}\|\sin\gamma)^2 = 0\\
&\frac{k}{2\pi}\sin\theta = \beta\|\bar{\bu}\|\sin\gamma \stepcounter{equation}\tag{\theequation}\label{Sine Relation}
\end{align*}
By the initialization of $\beta$ and the fact that $\|\bar{\bu}\| < 1$, we have
\[ \frac{k}{2\pi}\sin\theta < \frac{k}{2\pi}\sin\delta \]
this implies $\theta < \delta$.

Finally, expression (\ref{EquilibriumCondition}) can also be written as

\begin{equation}\label{ParallelCondition2}
\left( \bw^* - \frac{2\pi}{k} \beta (\bar{\bw}-\bar{\bu}) \right) \parallelsum \bar{\bw}
\end{equation}

From expression \eqref{ParallelCondition2}, we see that $\bw^*$, after subtracting some vector whose signs agree with $\bar{\bw}$, and whose non-zero components have the same magnitude $\frac{2\pi}{k}\lambda$, is parallel to $\bar{\bw}$. This implies $\bar{\bw}$ is some soft-thresholded version of $\bw^*$, modulo normalization. Moreover, since $\left\|\frac{2\pi}{k}\beta(\bar{\bw}-\bar{\bu})\right\| \leq \frac{2\pi}{k}\lambda\sqrt{d}$, for small $\lambda$ such that $\frac{2\pi}{k}\lambda\sqrt{d} < 1$, we must have
\[ \theta\left(\bw^* - \frac{2\pi}{k} \beta (\bar{\bw}-\bar{\bu}), \bar{\bw}\right) = 0 \]
On the other hand,
\begin{align*}
\left\| \bw^* - \frac{2\pi}{k} \beta (\bar{\bw}-\bar{\bu}) \right\|
&\geq \|\bw^*\| - \left\|\frac{2\pi}{k}\beta(\bar{\bw}-\bar{\bu})\right\| \\
&\geq 1 - \frac{2\pi}{k}\lambda\sqrt{d}
\end{align*}
therefore,
\begin{equation}\label{Theorem 1 result}
\bw^* - \frac{2\pi}{k} \beta (\bar{\bw}-\bar{\bu}) = C \bar{\bw}
\end{equation}
for some constant $C$ such that $0 < C \leq \frac{k}{k-2\pi\lambda\sqrt{d}}$.\\

Finally, consider the equilateral triangle with sides $\bw^*,\bar{\bw}$, and $\bw^* - \bar{\bw}$. By the law of sines,
\[ \frac{\|\bw^*-\bar{\bw}\|}{\sin\theta} = \frac{\|\bw^*\|}{\sin\theta(\bar{\bw},\bw^*-\bar{\bw})}
= \frac{1}{\sin\theta(\bar{\bw},\bw^*-\bar{\bw})} \]
as $\theta$ is small, $\theta(\bar{\bw},\bw^*-\bar{\bw})$ is near $\frac{\pi}{2}$. We can assume $\sin\theta(\bar{\bw},\bw^*-\bar{\bw}) \geq \frac{1}{2}$. Together with expression \eqref{Sine Relation}, we have
\[ \|\bw^* - \bar{\bw}\| \leq 2\sin\theta 
= \frac{\pi\beta\|\bar{\bu}\|\sin\gamma}{k}
\leq \frac{\pi\beta\sin\gamma}{k} 
\leq \frac{1}{2}\sin\delta\sin\gamma \]
The bound on $\|\bw^*-\bar{\bu}\|$ follows directly from triangle inequality.
\end{proof}

Combining Lemmas \ref{gradient lemma} - \ref{properties of limit point}, Theorem \ref{Main result} is proved.\qed

\subsection{Proof of Corollary}
\begin{lem}\cite{Zhang1}\label{TL1argmin}
Let 
\[ f_{\lambda,x}(y) = \frac{1}{2}(y-x)^2 + \lambda\,  \rho_a(y),  \]
\[ g_\lambda(x) = sgn(x)\left\{\frac{2}{3}(a+|x|)\cos\left(\frac{\phi(x)}{3}\right)-\frac{2a}{3} + \frac{|x|}{3}\right\} \]
where $\phi(x) = \arccos\left(1-\frac{27\lambda a(a+1)}{2(a+|x|)^3}\right)$.
Then $y_\lambda^*(x) = \arg\min_y f_{\lambda,x}(y)$ is the T$\ell_1$ thresholding, equal to $g_\lambda(x)$ if $|x|> t$; zero elsewhere.
Here $t=\lambda\frac{a+1}{a}$ if $\lambda \leq \frac{a^2}{2(a+1)}$;
$ t = \sqrt{2\lambda(a+1)} - \frac{a}{2}$, elsewhere.
\end{lem}

\begin{lem}\cite{Blumensath2}\label{l0argmin}
Let $f_{\lambda,x}(y) = \frac{1}{2}(y-x)^2 + \lambda\,  \|y\|_0. $
Then $y_\lambda^*(x) = \arg\min_y f_{\lambda,x}(y)$ is the $\ell_0$ hard  thresholding $y^*_\lambda(x) = x$, if $|x| > \sqrt{2\lambda}$; zero elsewhere. 
\end{lem}

We proceed by an outline similar to the proof of Theorem \ref{Main result}:

Step 1. First we show that $L_{\beta,T\ell_1}(\bu^t,\bw^t)$ and $L_{\beta,0}(\bu^t,\bw^t)$ both decrease under the update of $\bu^t$ and $\bw^t$. To see this, notice that the update on $\bu^t$ decreases $L_{\beta,T\ell_1}(\bu^t,\bw^t)$ and $L_{\beta,0}(\bu^t,\bw^t)$ by definition. Then, for a fixed $\bu = \bu^{t+1}$, the update on $\bw^t$ decreases $L_{\beta,T\ell_1}(\bu^t,\bw^t)$ and $L_{\beta,0}(\bu^t,\bw^t)$ by a similar argument to that found in Theorem \ref{Main result}.

Step 2. Next, we show $\theta(\bw^t,\bw^*) \leq \pi-\delta$, for some $\delta > 0$, for all $t$, with initialization $\theta(\bw^0,\bw^*) = \pi-\delta$. For $L_{\beta,T\ell_1}(\bu^t,\bw^t)$, by Lemma \ref{TL1argmin}, we have
\[ \bu^{t+1} = (g_{\lambda/\beta}(w_1^t),g_{\lambda/\beta}(w_2^t),...,g_{\lambda/\beta}(w_d^t)) \]
And for $L_{\beta,0}(\bu^t,\bw^t)$ , by Lemma \ref{l0argmin},
\[ \bu^{t+1} = (w_1^t\chi_{\{|w_1^t|\geq t\}}, w_2^t\chi_{\{|w_2^t|\geq t\}},...) \]
In both cases, each component of $\bu^{t+1}$ is a thresholded version of the corresponding component of $\bw^t$. This implies $\theta(\bu^{t+1},\bw^t) \leq \frac{\pi}{2}$, and thus the argument in Theorem \ref{Main result} follows through, and we have $\theta(\bw^t,\bw^*) \leq \pi-\theta$, for all $t$.

Step 3. Finally, the equilibrium condition from equation $\eqref{ParallelCondition}$ still holds for the critical point, and a similar argument shows that $\theta(\bar{\bw},\bw^*) < \delta$. \qed

\section{Numerical Experiment}
Let $(k,d)=(20,50)$, $\beta=10^{-3}$, $\eta =10^{-5}$, $\|\bw^*\|_0=s \ll d$, $\bw^* =s^{-1/2} (1,\cdots,1,0,\cdots,0)$. Table \ref{tab:1} shows result by RVSCGD with $\ell_0$ penalty,   $\lambda=10^{-4}$, random start in $\bw$. The sparsity $s$ is exactly recovered in $\bar{\bu}$, with small angular errors and loss values. The angle descent vs. iterations is seen in Fig. \ref{angleplot}. 

\begin{table}[!ht]
\caption{RVSCGD regression: angular errors, loss and sparsity.}
\label{tab:1}
\vskip 0.1in
\begin{center}
\begin{small}
\begin{tabular}{lccccc}
  \toprule
  measure & s=2 & s=4 & s=6&s=8&s=10\\
  \midrule
  $\theta(\bar{\bu},\bw^*)$ &  0.0076 & 0.0122  & 0.0126 & 0.0133& 0.0143\\
  $\theta(\bar{\bw},\bw^*)$ & 0.0492  & 0.0498 & 0.0499 & 0.0497&0.0493\\
  $f(\bar{\bu})$ & 0.0243 & 0.0388 & 0.0402 & 0.0423& 0.0456\\
  $\|\bar{\bu}\|_0$ & 2 & 4 & 6 &8&10\\
  \bottomrule
\end{tabular}
\end{small}
\end{center}
\end{table}

\begin{figure}[!ht]
\begin{center}
\centerline{\includegraphics[width=0.8\columnwidth]{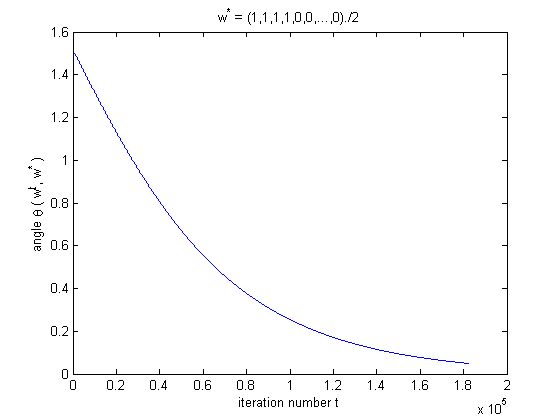}}
\caption{Descent of angle between $\bw^t$ and $\bw^*$ vs iteration number ($s=4$), in line with Lemma \ref{angle decreases lemma}.}
\label{angleplot}
\end{center}
\vskip -0.2in
\end{figure}

\section{Conclusion}
We introduced a variable splitting coarse gradient descent method to learn a one-hidden layer neural network with sparse weight and binarized activation in a regression setting. The proof is based on the descent of a  Lagrangian function and the angle between the sparse and true weights, and applies to $\ell_1$, $\ell_0$ and T$\ell_1$ sparse penalties. We plan to extend our work to a classification setting in the future.  

\section{Acknowledgement}
The work was partially supported by NSF grant IIS-1632935. The authors thank Dr. Penghang Yin for the helpful comments.

\bibliographystyle{plain}
\bibliography{tdbib}

\end{document}